\documentclass[a4paper]{amsart}
\usepackage{dsfont}
\usepackage{amsmath,amsthm,amssymb,hyperref,mathrsfs}
\usepackage{aliascnt}
\usepackage{lmodern}
\usepackage[T1]{fontenc}
\usepackage[textsize=footnotesize]{todonotes}

\usepackage{xcolor}
\definecolor{dblue}{rgb}{0,0,0.70}
\hypersetup{
	unicode=true,
	colorlinks=true,
	citecolor=dblue,
	linkcolor=dblue,
	anchorcolor=dblue
}

\makeatletter
\expandafter\g@addto@macro\csname th@plain\endcsname{%
		\thm@notefont{\bfseries}
	}%
\expandafter\g@addto@macro\csname th@remark\endcsname{%
		\thm@headfont{\bfseries}
	}%
\makeatother

\newtheorem{theorem}
{Theorem}[section]
\newtheorem*{theorem*}{Theorem}
\newtheorem*{mtheorem}{Main Theorem}

\newaliascnt{lemma}{theorem}
\newtheorem{lemma}[lemma]{Lemma}
\aliascntresetthe{lemma}
\newtheorem*{lemma*}{Lemma}

\newaliascnt{fact}{theorem}

\aliascntresetthe{fact}

\newaliascnt{proposition}{theorem}
\newtheorem{proposition}[proposition]{Proposition}
\aliascntresetthe{proposition}
\newtheorem*{proposition*}{Proposition}

\newaliascnt{corollary}{theorem}
\newtheorem{corollary}[corollary]{Corollary}
\aliascntresetthe{corollary}

\theoremstyle{remark}

\newaliascnt{remark}{theorem}
\newtheorem{remark}[remark]{Remark}
\aliascntresetthe{remark}
\newaliascnt{question}{theorem}
\newtheorem{question}[question]{Question}
\aliascntresetthe{question}
\newaliascnt{conjecture}{theorem}

\aliascntresetthe{conjecture}

\newtheorem*{question*}{Question}

\newaliascnt{definition}{theorem}
\newtheorem{definition}[definition]{Definition}
\aliascntresetthe{definition}

\newaliascnt{example}{theorem}

\aliascntresetthe{example}

\renewcommand{\restriction}{\mathbin\upharpoonright}

\newcommand{\axiom}[1]{\mathsf{#1}}
\newcommand{\ZFC}{\axiom{ZFC}}

\newcommand{\DC}{\axiom{DC}}
\newcommand{\ZF}{\axiom{ZF}}
\newcommand{\ZFA}{\axiom{ZFA}}

\newcommand{\GCH}{\axiom{GCH}}

\newcommand{\HS}{\axiom{HS}}

\DeclareMathOperator{\cf}{cf}
\DeclareMathOperator{\dom}{dom}

\DeclareMathOperator{\sym}{sym}

\DeclareMathOperator{\fix}{fix}

\DeclareMathOperator{\id}{id}
\DeclareMathOperator{\aut}{Aut}
\DeclareMathOperator{\Col}{Col}
\DeclareMathOperator{\Add}{Add}

\newcommand{\forces}{\mathrel{\Vdash}}

\newcommand{\nto}{\mathrel{\nrightarrow}}

\newcommand{\incompatible}{\mathrel{\bot}}
\newcommand{\compatible}{\mathrel{\|}}
\newcommand{\power}{\mathcal{P}}

\newcommand{\PP}{\mathbb P}
\newcommand{\QQ}{\mathbb Q}

\newcommand{\BB}{\mathbb B}
\newcommand{\CC}{\mathbb C}

\newcommand{\cA}{\mathcal A}

\newcommand{\cI}{\mathcal I}
\newcommand{\cL}{\mathcal L}

\newcommand{\sF}{\mathscr F}
\newcommand{\sG}{\mathscr G}

\newcommand{\ccc}[1]{%
  \ifthenelse{\equal{#1}{}}{\mathrm{CCC}}{\mathrm{CCC}_#1}%
}

\newcommand{\1}{\mathds 1}

\newcommand{\tup}[1]{\langle#1\rangle}

\author{Asaf Karagila}
\thanks{The first author was supported by a UKRI Future Leaders Fellowship [MR/T021705/1].}
\address[Asaf Karagila]{}
\address{School of Mathematics,
    University of Leeds.
    Leeds, LS2~9JT, UK}
\email{karagila@math.huji.ac.il}
\urladdr{http://karagila.org}

\author{Noah Schweber}
\address[Noah Schweber]{}
\email{nschweber@proofschool.org}

\date{June 9, 2022}

\title{Choiceless Chain Conditions}

\begin{document}
\begin{abstract}
Chain conditions are one of the major tools used in the theory of forcing. We say that a partial order has the countable chain condition if every antichain (in the sense of forcing) is countable. Without the axiom of choice antichains tend to be of little use, for various reasons, and in this short note we study a number of conditions which in $\ZFC$ are equivalent to the countable chain condition.
\end{abstract}
\maketitle              
\section{Introduction}
Paul Cohen developed the method of forcing to prove that the Continuum Hypothesis is not provable from $\ZFC$. He did so by showing that one can add $\aleph_2$ new real numbers to a model of $\ZFC$, but in order to conclude that indeed the Continuum Hypothesis is false in the resulting model, Cohen needed to show that the ordinal which was $\omega_1$ in the ground model did not become countable when we added those new real numbers. In \cite{Cohen:1964} Cohen shows that any set of pairwise incompatible conditions (or \textit{antichain}) is countable, and then uses this lemma to conclude that if $\alpha$ and $\beta$ are two ordinals that are in bijection in the generic extension, then they must also be in bijection in the ground model as well.

This property, now known as the \textit{countable chain condition}, was generalised in several ways (both $\kappa$-c.c.\ as well as Axiom A and properness are, in some sense, generalisations). The preservation of cardinals and cofinalities is so deeply ingrained into the study of forcing that any new definition of a forcing will come with the study of its chain conditions, either directly by proving that no antichain can be of a certain size, or by showing that certain cardinals or cofinalities change in the generic extensions, and thus putting an upper bound on the chain condition of the forcing.

If we choose to study forcing over models of $\ZF$, however, this changes drastically. Despite the basic machinery of generic extensions remaining the same, antichains and closure conditions play a very different role. For one, ``every partial order contains a maximal antichain'' is itself equivalent to the axiom of choice over $\ZF$.\footnote{This equivalence, unlike ``every partial order contains a maximal chain'', requires the axiom of regularity. See Chapter~9 of \cite{Jech:AC} for details.} But antichains in partial orders are not the same as antichains in forcing, since ``incomparable'' and ``incompatible'' are very different notions in general. Arnold Miller, in his note \cite[p.~2]{Miller:2011}, remarks that while a certain partial order does not have maximal antichains, it is trivial from a forcing perspective. In the proof that ``every partial order contains a maximal antichain'' implies the axiom of choice, one resorts to fairly trivial partial orders as well, as far as forcing is concerned. These are forcings that when considering their separative quotients are just atomic forcings.

In this paper we study the notion of chain conditions in $\ZF$ in somewhat less trivial way. We focus on the countable chain condition in particular, since it is by far the most useful one. We consider three different versions, ranging from ``every maximal antichain is countable'' to ``every predense set contains a countable predense subset'', and we show that $\ZF+\DC$ will not prove any of the non-trivial implications between these versions. We also use two external definitions of the countable chain condition,\footnote{External in the sense that they require more than just the subsets of the forcing.} and we compare them against our three proposed ones.

We use this to argue that the strongest definition is the one that should be used, at least in the context of partial orders, as it is consistent with $\ZF+\DC$ that a partial order satisfies that every antichain is countable, but it collapses $\omega_1$. If we are willing to give up on $\DC$ and relax the version of countable chain condition we can also add $\sigma$-closure. In contrast, $\ZFC$ proves that a forcing which is both $\kappa$-c.c.\ and $\kappa$-closed is atomic. On the other hand, the strongest version of the countable chain condition cannot change cofinalities or collapse cardinals above $\omega_1$.

We also include yet another transfer theorem that allows us to ``clone'' a structure in the ground model in a way that allows us to control the subsets of the cloned version in a very precise manner. These sort of theorems are used often to prove theorems in $\ZF$.

At the end of this paper we include a list of open questions which we think are crucial for pushing further the study of forcing over models of $\ZF$, putting us one step closer to the goal of understanding how forcing changes the cardinal structure of models of $\ZF$.
\section{Preliminaries}
Suppose that $X$ and $Y$ are two sets, we want to look at permutations of $X\times Y$ which preserve $X$-sections. In other words, we want to consider permutations $\pi$ of $X\times Y$, such that whenever $y,z\in Y$, then $\pi(x,y)$ and $\pi(x,z)$ will appear in the same $X$ section, that is, if $\pi(x,y)=\tup{x_0,y'}$ and $\pi(x,z)=\tup{x_1,z'}$, then $x_0=x_1$. In this case, $\pi$ defines a permutation of $X$ by its action on the $X$ section, which we denote by $\pi^*$, and for every $x\in X$, $\pi_x$ is the permutation of $Y$ defined by $\pi_x(y)=z$ if and only if $\pi(x,y)=\tup{x',z}$ for some $x'\in X$. This leads us to the \textit{wreath product}. Let $G\subseteq\sym(X)$ and $H\subseteq\sym(Y)$ be two groups. The wreath product of $G$ and $H$, denoted by $G\wr H$, is the set of all permutations $\pi$ of $X\times Y$ which preserve $X$-section such that $\pi^*\in G$ and for all $x\in X$, $\pi_x\in H$.

We say that a set is \textit{well-orderable} if it can be well-ordered, in which case its cardinal is the least ordinal which is in bijection with the set. The cardinals of infinite well-orderable sets are the $\aleph$ numbers. If $X$ is not well-orderable, we define its cardinal, denoted by $|X|$, to be its \textit{Scott cardinal}, namely the set \[\{Y\in V_\alpha\mid\exists f\colon X\to Y\text{ a bijection}\},\] where $\alpha$ is the least ordinal for which the set is non-empty. We will use Greek letters such as $\kappa,\lambda$ and $\mu$ to denote $\aleph$ numbers exclusively.

The only weak versions of the axiom of choice which we will use in this work are variants of the Principle of Dependent Choice. For an $\aleph$ number $\lambda$ we define $\DC_\lambda$ to be the statement ``Every $\lambda$-closed tree has a maximal element or a chain of type $\lambda$'', and we define $\DC_{<\kappa}$ to mean $\forall\lambda<\kappa,\DC_\lambda$. If $\lambda=\omega$, then $\DC_\lambda$ is written simply as $\DC$.

\subsection{Forcing and antichains}
We say that $\PP$ is a \textit{notion of forcing}, or simply a forcing, if it is a preordered set with a maximum, $\1_\PP$. Throughout the paper it will often be beneficial to consider forcing notions with $\1$ removed from the partial order, but this will always be a temporary measure.

If $p,q\in\PP$, we say that $q$ \textit{extends} $p$, or that it is a \textit{stronger} condition if $q\leq_\PP p$. Two conditions $p$ and $q$ are \textit{compatible} (abbreviated as $p\compatible q$) if they have a common extensions, and otherwise they are incompatible (written as $p\incompatible q$).

We say that a subset $A\subseteq\PP$ is an \textit{antichain} if any two distinct conditions in $\PP$ are incompatible. Finally, we say that $D\subseteq\PP$ is \textit{predense} if for any $p\in\PP$ there is some $q\in D$ which is compatible with $p$. We note that an antichain is predense if and only if it is a maximal antichain.

For a subset $A$ of $\PP$, we denote by $A^\perp$ the set $\{p\in\PP\mid\forall q\in A,p\perp q\}$. Easily, $A\cup A^\perp$ is predense for any set $A$, and $A$ is predense if and only if $A^\perp=\varnothing$.

We say that a forcing $\PP$ satisfies the \textit{$\kappa$-chain condition} (or $\kappa$-c.c.), at least in the standard context of $\ZFC$, if given any subset of $\PP$ of size $\kappa$, at least two of its members are compatible. If $\kappa=\aleph_1$ we refer to it as the countable chain condition, or c.c.c. If $\PP$ is a c.c.c.\ forcing, then it does not change cofinalities, and in particular does not collapse cardinals.\footnote{This easily follows from \autoref{cor:ccc3-is-well-behaved} in the context of $\ZFC$.}

If $\{\dot x_i\mid i\in I\}$ is a family of $\PP$-names, usually a set, we want to make it into a name that is interpreted as the set of the interpretations. If $\1_\PP$ is present, this is easily done by defining $\{\dot x_i\mid i\in I\}^\bullet=\{\tup{\1_\PP,\dot x_i}\mid i\in I\}$. This notation extends naturally to ordered pairs, functions, sequences, etc. We also note that the usual naming scheme for ground model sets can be made more compact using this notation since $\check x=\{\check y\mid y\in x\}^\bullet$.

Finally, given a family of forcing notions, $\{\PP_i\mid i\in I\}$, we define the \textit{lottery sum} $\bigoplus_{i\in I}\PP_i$ as the forcing $\{\1\}\cup\{\{i\}\times\PP_i\}$ with the order given by $\tup{j,q}\leq\tup{i,p}$ if and only if $i=j$ and $q\leq_i p$, and of course $\1$ is the maximum. This can be viewed as an iteration of first choosing $i\in I$ with an atomic forcing, then forcing with that $\PP_i$. The idea is to let the generic filter choose which forcing we are using.

\subsection{Symmetric extensions}
Forcing is generally done in the context of $\ZFC$, in part because the usefulness of antichains, and since forcing cannot violate the axiom of choice, this method alone is not all that useful when trying to prove consistency results between $\ZF$ and $\ZFC$. We can extend this method to accommodate the failure of the axiom of choice in a rather controlled manner by introducing additional structure to the forcing. The additional structure allows us to identify a class of names which defines a model of $\ZF$ between the ground model and the generic extension.

Let $\PP$ be a notion of forcing. If $\pi$ is an automorphism of $\PP$, then $\pi$ acts on the $\PP$-names in the following recursive manner: \[\pi\dot x=\{\tup{\pi p,\pi\dot y}\mid\tup{p,\dot y}\in\dot x\}.\] The forcing relation, which is defined from the order of $\PP$ is also respected by this action as shown by the Symmetry Lemma.
\begin{lemma}\label{symmetry-lemma}
  Suppose that $p\in\PP$, $\dot x$ is a $\PP$-name, and $\pi\in\aut(\PP)$, then
  \[p\forces\varphi(\dot x)\iff\pi p\forces\varphi(\pi\dot x)\qed\]
\end{lemma}

Let $\sG$ be a fixed subgroup of $\aut(\PP)$. We say that $\sF$ is a \textit{filter of subgroups on $\sG$} if it is a filter on the lattice of subgroups. Namely, it is a non-empty collection of subgroups of $\sG$ which is closed under supergroups and finite intersections. We will also assume that the trivial group is not in $\sF$. We say that $\sF$ is a \textit{normal} filter of subgroups if whenever $H\in\sF$ and $\pi\in\sG$, then $\pi H\pi^{-1}\in\sF$ as well. We say that $\tup{\PP,\sG,\sF}$ is a \textit{symmetric system} if $\PP$ is a notion of forcing, $\sG$ is a group of automorphisms of $\PP$, and $\sF$ is a normal filter of subgroups on $\sG$.

We denote by $\sym_\sG(\dot x)$ the subgroup $\{\pi\in\sG\mid\pi\dot x=\dot x\}$, and we say that $\dot x$ is an \textit{$\sF$-symmetric} name when $\sym_\sG(\dot x)\in\sF$. If this notion holds hereditarily for all names which appear in the transitive closure of $\dot x$, then we say that $\dot x$ is a \textit{hereditarily $\sF$-symmetric} name. We denote by $\HS_\sF$ the class of all hereditarily $\sF$-symmetric names.

\begin{theorem*}[Lemma~15.51 in \cite{Jech:ST2003}]
Suppose that $\tup{\PP,\sG,\sF}$ is a symmetric system, and let $G\subseteq\PP$ be a $V$-generic filter. The class $\HS_\sF^G=\{\dot x^G\mid\dot x\in\HS_\sF\}$ is transitive class model of $\ZF$ which lies between $V$ and $V[G]$.
\end{theorem*}
The class $\HS_\sF^G$ in the theorem is also known as a \textit{symmetric extension}. To learn more about symmetric extensions in general, as well as their iterations we direct the reader to \cite{Grigorieff:1975}, and more recent papers such as \cite{Usuba:LS} and \cite{Karagila:2019}.

\section{Subset control}
We ultimately want to construct models of $\ZF$ where a certain partial order will fail to satisfy some definition of c.c.c.
\begin{definition}\label{def:sym-copy}
  Suppose that $M$ is a structure in some first-order language $\cL$ and $V\subseteq W\subseteq V[G]$ is a symmetric extension of $V$. A \textit{symmetric copy} of $M$ is a structure $N\in W$ such that $V[G]\models M\cong_\cL N$.\footnote{The definition in fact captures general logics, but we will only need it for first-order logic in this paper.}
\end{definition}

Our goal, therefore, is to define a partial order in the ground model, with a suitable group of automorphisms, and use these to define a symmetric extension in which there is a symmetric copy of our partial order whose subsets are exactly those that are fixed pointwise by a large group of automorphisms (with the appropriate definition of large, usually meaning ``fixing pointwise a small number of points'').

These sort of considerations are not new. Plotkin \cite{Plotkin:1969}, Hodges \cite{Hodges:1974}, and others already proved theorems of this sort before including many ad hoc versions, and if we want to be strict, we may also include results involving permutations models of $\ZFA$ along with the Jech--Sochor transfer theorem (see \cite[Ch.~6]{Jech:AC} for details). The first author has proved a similar theorem as well, in a more general setting that allows the preservation of $\DC_{<\kappa}$ under suitable conditions \cite{Karagila:DC}. For these reasons we will only outline the proof of the theorem.

Let $\cL$ be a first-order language and let $M$ be an $\cL$-structure. Given a group $\sG\subseteq\aut(M)$ and an ideal of subsets of $M$, $\cI$, we say that a subgroup of $\sG$ is large if it contains $\fix(A)=\{\pi\in\sG\mid \pi\restriction A=\id\}$ for some $A\in\cI$. We fix such $\cL,M,\sG$ and $\cI$. We will say that a $X\subseteq M$ is \textit{stable} under an automorphism $\pi$, if $\pi``X=X$, and we say it is stable under a group $H$ if it is stable under every $\pi\in H$.

\begin{theorem}\label{thm:transfer}
Let $\cL,M,\sG,\cI$ be as above. There is a symmetric extension of the universe in which there is a symmetric copy of $M$, $N$, such that every subset of $N^k$ in the symmetric extension is a symmetric copy of a subset of $M^k$ that is stable under a large group of automorphisms.

  Moreover, if $\cI$ is $\kappa$-complete, then we can assume that $\DC_{<\kappa}$ holds in the extension. We can also require that $M$ does not have any new subsets, although that is not relevant to the proof.
\end{theorem}
\begin{proof}[Outline of Proof]
  We will only outline the case for the case of $k=1$. Let $\lambda\geq\omega$ be your favourite regular cardinal (e.g., $\omega$ if we do not wish to collapse cardinals without additional assumptions such as $\GCH$; or $|M|^+$ if we wish to not add new subsets of $M$), and consider the forcing $\PP=\Add(\lambda,M\times\lambda)$ whose conditions are partial functions $p\colon M\times\lambda\times\lambda\to 2$ such that $|p|<\lambda$).

  We let $\sG\wr\sym(\lambda)$ act on $\PP$ in the natural way: $\pi p(\pi(m,\alpha),\beta)=p(m,\alpha,\beta)$, and let $\sF$ be the filter of subgroups $\{\fix(A\times B)\mid A\in\cI, B\in[\lambda]^{<\lambda}\}$.

  For $m\in M,\alpha<\lambda$ let
  \begin{enumerate}
  \item $\dot x_{m,\alpha}=\{\tup{p,\check\beta}\mid p(m,\alpha,\beta)=1\}$,
  \item $\dot a_m=\{\dot x_{m,\alpha}\mid\alpha<\lambda\}^\bullet$, and
  \item $\dot N=\{\dot a_m\mid m\in M\}^\bullet$, we define similar names the interpretation of $\cL$. E.g., if $\cL$ has a partial order symbol $\leq$, then we define $\dot\leq=\{\tup{a_m,a_n}^\bullet\mid m\leq n\}^\bullet$.
  \end{enumerate}
  It is easy to see that this is indeed defining a symmetric copy of $M$. We say that a subset of $N$ is stable under $\pi$ if in the full generic extension, where we have the natural identification $m\mapsto x_m$, it is stable under the transport of $\pi$ to $N$. In other words, if $p\forces\dot X\subseteq\dot N$, then $p$ forces that $\dot X$ is stable under $\pi$ if $p\forces\dot a_m\in\dot X\leftrightarrow\dot a_{\pi^*m}\in\dot X$ for all $m\in M$.\footnote{Recall that $\pi^*$ is the projection of $\pi$ to $\sG$.} So in the generic extension, we can transfer the notion of ``stable'' to $N$, and so $X$ is stable under a large group of automorphisms if there is some $A\in\cI$ such that $X$ is stable under $\fix(A)$.

  The following lemma is the key technical part of the proof:

  \begin{lemma}\label{lemma:homogeneity}
    Given any two conditions $p,q$ there is a permutation $\pi$ such that $\pi^*=\id$ and $\pi p$ is compatible with $q$.
  \end{lemma}
  \begin{proof}[Proof of Lemma]
    Let $A\subseteq M$ be the set of $m\in M$ such that $p(m,\alpha,\beta)\neq q(m,\alpha,\beta)$ for some $\alpha,\beta<\lambda$.

    Then for each $m\in A$ we have that $|\{\alpha<\lambda\mid \exists\beta\, \tup{m,\alpha,\beta}\in\dom p\cup\dom q\}|<\lambda$. Therefore there is a permutation of $\lambda$, $\pi_m$ such that \[\{\pi_m\alpha\mid\exists\beta\, \tup{m,\alpha,\beta}\in\dom p\}\cap\{\alpha\mid\exists\beta\,\tup{m,\alpha,\beta}\in\dom q\}=\varnothing.\] We can now take $\pi$ to be $\pi^*=\id$ and $\pi_m$ to be our chosen permutation as needed, then $\dom\pi p\cap\dom q=\varnothing$. This construction can be modified to simply move the coordinates which are in disagreement, or to preserve any small subset of the domain if necessary.
  \end{proof}

  Suppose now that $\dot X\in\HS$ and $p\forces\dot X\subseteq\dot N$, and let $A\times B$ be such that $\fix(A\times B)\subseteq\sym(\dot X)$. By the homogeneity lemma above, we may assume that $\dom p\subseteq A\times B\times\lambda$, since we are only interested in statements of the form $\dot a_m\in\dot X$ and their negation. Suppose now that $q\leq p$ and $q\forces\dot a_m\in\dot X$, if we show that for every $\pi\in\fix(A\times B)$, $q\forces\pi\dot a_m\in\dot X$, then we are done since $\pi\dot a_m=\dot a_{\pi^*m}$, and that would mean that $p$ forces that $\dot X$ is stable under $\fix(A)$.

  Let $\pi\in\fix(A\times B)$, then $\pi q\forces\pi\dot a_m\in\pi\dot X$, and since $\pi\dot X=\dot X$ we can omit it from that part. Moreover, by the lemma above we can find $\tau\in\fix(A\times B)$ such that $\tau^*=\id$ and $\tau\pi q$ is compatible with $q$. But now $\tau\pi q\forces\tau\pi\dot a_m\in\dot X$, and since $\tau\pi\dot a_m=\dot a_{\tau^*\pi^*m}=\dot a_{\pi^*m}=\pi\dot a_m$ we get that $\tau\pi q\forces\pi\dot a_m\in\dot X$. Since $q$ was an arbitrary extension of $p$, this holds for every extension of $q$ itself. Therefore no extension of $q$ can force $\pi\dot a_m\notin\dot X$, so $q\forces\pi\dot a_m\in\dot X$ as wanted.

  Finally, if $\cI$ is $\kappa$-closed, then by choosing $\lambda\geq\kappa$ we have that $\sF$ is $\kappa$-complete and $\PP$ is $\kappa$-closed and therefore by \cite{Karagila:DC} we preserve $\DC_{<\kappa}$.\footnote{We can be more careful in our construction to even allow $\lambda=\omega$ while still preserving $\DC_{<\kappa}$. But since we will not need that much sophistication in our use of this theorem it is easier to make the assumption that $\lambda\geq\kappa$.} And indeed, if we wish for $M$ to not have new subsets we only need to take $\lambda=|M|^+$.\end{proof}

\begin{remark}
  The proof above raises the obvious question: if $X\subseteq N$ is in the generic extension and it is stable under a large group of automorphisms, is it in the symmetric extension? In the case where $X$ is a copy of a subset of $M$ in the ground model the answer is easily positive: $\{\dot a_m\mid m\in X_*\}^\bullet$, where $X_*\subseteq M$ is the relevant subset, is a symmetric name if and only if $X_*$ is stable under a large group of automorphisms.

  Therefore, in the case where no new subsets of $M$ are added this question receives a very easy positive answer. In the case where new subsets of $M$ are added we run into the problem that if $\dot X$ is a name for a subset of $N$, then $\pi\dot X$ is not necessarily the same as $\pi``X$ (understood as the transport of $\pi$ in the generic extension). The proof requires us to define a name in $\HS$ and show that it is forced to be equal to $\dot X$, but in order to do so we need either equality between $\pi\dot X$ and $\pi``X$, or at the very least that $\dot X$ is symmetric to begin with, in which case there's nothing to do anyway.

  Under additional assumptions of homogeneity on $M$ we can indeed get these sort of equivalences but finding the exact condition seems like a result beyond the scope that is necessary here. And so we decide to add one more piece of wood into the roaring fire of ad-hoc folklore theorems instead.
\end{remark}

\section{What do you mean by ``a chain condition''?}
\begin{definition}
  Let $\PP$ be a notion of forcing. We define the following properties:
  \begin{description}
  \item[$\ccc1$] Every maximal antichain in $\PP$ is countable.
  \item[$\ccc2$] Every antichain in $\PP$ is countable.\footnote{Equivalently, every uncountable subset of $\PP$ contains two compatible conditions.}
  \item[$\ccc3$] Every predense subset of $\PP$ contains a countable predense subset.
  \end{description}
\end{definition}
When we write $\ccc{i}\to\ccc{j}$, we mean that for every notion of forcing $\PP$, if $\ccc{i}(\PP)$ holds, then $\ccc{j}(\PP)$ holds. Likewise, $\ccc{i}\nto\ccc{j}$ means that there exists a notion of forcing for which the implication does not hold.

Assuming $\ZFC$, all the variants mentioned above are equivalent and the ``standard definition'' is $\ccc2$.
\begin{proposition}\label{prop:zf-implications}
  The following implications are provable in $\ZF$:
  \[\ccc3\to\ccc2\to\ccc1.\]
\end{proposition}
\begin{proof}
  $\ccc3\to\ccc2$: Let $A$ be an antichain, and let $D=A\cup A^\perp$. Since $D$ is predense, it has a countable predense subset, $D'$. Because every condition in $D$ is either incompatible with all conditions in $A$ or it belongs to $A$, it must be that $A\subseteq D'$. Therefore $A$ is countable. $\ccc2\to\ccc1$ is trivial.
\end{proof}
\begin{proposition}\label{prop:ccc2-to-ccc3-under-strong-dc}
If $\DC_{\omega_1}$ holds, then $\ccc2\to\ccc3$.
\end{proposition}
\begin{proof}
  Suppose that $\PP$ is a $\ccc2$ forcing notion and $D$ is a predense set, we will find a countable $D_0\subseteq D$ which is also predense. We will begin with the case that $D$ is in fact a dense open set. In this case, we recursively construct an increasing sequence of antichains, $A_\alpha$. Taking $A_0=\varnothing$ and $A_\alpha=\bigcup_{\beta<\alpha}A_\beta$ when $\alpha$ is a limit ordinal. In the successor case, simply choose a point $p\in D\cap A_\alpha^\perp$ if such $p$ exists and let $A_{\alpha+1}=A_\alpha\cup\{p\}$. If no such $p$ exists, then $A_\alpha$ is a maximal antichain, since $D$ was dense, and we simply let $A_{\alpha+1}=A_\alpha$.

  Using $\DC_{\omega_1}$, the process can continue to define $A_\alpha$ for $\alpha<\omega_1$ and so $A=\bigcup_{\alpha<\omega_1}A_\alpha$ is an antichain in $\PP$ which is contained in $D$. But since $\PP$ was $\ccc2$, it must be that $A$ is countable, so the process must have stabilised at some stage and therefore $A$ is maximal and is predense.

  Assume now that $D$ is not a dense open set. We can repeat the same process for its downwards closure, that is $\{q\in\PP\mid\exists p\in D: q\leq p\}$ to find a countable and maximal antichain $A$. Now, using $\DC_{\omega_1}$ we can choose for each $q\in A$ some $p_q\in D$ such that $q\leq p_q$, and such $p_q$ exists by the fact we took $A$ out of the downwards closure of $D$. Indeed, we may take up to countably many points for any such $q\in A$. Let $D_0$ be the countable set $\{p_q\mid q\in A\}$, then given any $p\in\PP$, there is some $q\in A$ such that $q$ and $p$ are compatible, and therefore $p_q$ is compatible with $p$ as well, so $D_0$ is predense as wanted.
\end{proof}
\begin{mtheorem}
No implication in \autoref{prop:zf-implications} is reversible in $\ZF+\DC$.
\end{mtheorem}
Each of these reverse implications requires a separate consideration. We will therefore prove them separately in the remainder of this section. We will rely, heavily, on \autoref{thm:transfer} and so it is enough to describe a separative partial order, a group of automorphisms, and an ideal of sets.
  \begin{proposition}\label{thm:ccc1}
$\ccc1\nto\ccc2$.
\end{proposition}
\begin{proof}
  Let $\QQ$ denote the Boolean completion of the Cohen forcing with its maximum element removed, and let $\PP$ be $\bigoplus_{\alpha<\omega_2}\QQ_\alpha$ with $\QQ_\alpha=\QQ$. Note that $\PP$ is a homogeneous forcing, given two conditions we can first permute $\omega_2$ to make sure that they come from the same index, and then apply the necessary automorphism of $\QQ$ making the two conditions compatible within that copy.

  To apply \autoref{thm:transfer} we consider $\PP$ with its full automorphism group and pointwise stabilisers of sets of size $\aleph_1$. If $A$ is a maximal antichain in $\PP$, then $A$ must be of size $\aleph_2$ or trivial, since if $\1_\PP\notin A$, then $A$ must meet each $\QQ_\alpha$, and therefore it has size $\aleph_2$. On the other hand, if an antichain has size $\aleph_2$, then it is not stable under any $\fix(E)$ for some $E$ of size $\aleph_1$. Let $A$ be such antichain, then since no $\QQ_\alpha$ has a maximum element of its own, there is some large enough $\alpha$ such that $\QQ_\alpha\cap A$ is non-empty and $\QQ_\alpha\cap E=\varnothing$. In that case, simply apply some automorphism which moves an element of $A\cap\QQ_\alpha$ to an element of $\QQ_\alpha\setminus A$, and such automorphism exists since $\QQ$ is homogeneous.

  Therefore the symmetric copy of $\PP$ is such where the only maximal antichain is trivial, but antichains of size $\aleph_1$ exists. And so we have a forcing that is $\ccc1$ but not $\ccc2$.
\end{proof}
\begin{proposition}\label{thm:ccc2}
$\ccc2\nto\ccc3$
\end{proposition}
\begin{proof}
  Let $\QQ$ denote a homogeneous c.c.c.\ partial order with its maximum element removed, and consider the lottery sum $\PP=\bigoplus_{\alpha<\omega_1}\QQ_\alpha$, where $\QQ_\alpha=\QQ$ for all $\alpha$. We then consider the automorphism group given by the full support product of $\aut(\QQ)$, acting on the summands pointwise. Finally, we let $\cI$ denote the ideal of countable subsets. Using \autoref{thm:transfer} we get a symmetric version of this partial order such that any symmetric subset, $A$, $\{\alpha\mid A\cap\QQ_\alpha\text{ is non-trivial}\}$ is co-countable. Here a subset is non-trivial if it is not the union of orbits. By homogeneity no trivial subset can be an antichain.


  Now we only need to find a predense set that cannot be reduced to a countable predense set, for example $\bigcup_{\alpha<\omega_1}\QQ_\alpha=\PP\setminus\{\1_\PP\}$. It is indeed predense, but it does not contain any countable predense set, since any countable subset is contained in a countable part of the sum and is therefore not predense.
\end{proof}

The constructions in all three cases satisfy the condition that the ideal of sets is countably closed and therefore $\DC$ holds. We can also, for the sake of clarity, assume that whatever partial orders we used in the ground model do not have any new subsets added, so they certainly preserve their c.c.c.
\section{External definitions and additional results}
\subsection{Mekler's c.c.c.: generic conditions for elementary submodels}
One of the most important concepts in forcing is that or properness, which is a generalisation of c.c.c. We say that a forcing $\PP$ is \textit{proper} if for all sufficiently large $\kappa$, whenever $M\prec H_\kappa$ is a countable elementary submodel such that $\PP\in M$, every $p\in\PP\cap M$ has an extension, $q$, in $\PP$ which is $M$-generic. That is to say, if $D\in M$ is a predense subset of $\PP$, then $D\cap M$ is predense below $q$.

Mekler showed in \cite{Mekler:1984} that the statement ``$\PP$ satisfies the countable chain condition'' is equivalent to requiring that not only $\PP$ is proper, but in fact every condition is $M$-generic in the definition of properness,\footnote{Equivalently, $\1_\PP$ is $M$-generic.} and we can now take this requirement of $M$-genericity as another definition of c.c.c., instead of talking about antichains and predense subsets. This, of course, reinforces the idea that properness is a generalisation of c.c.c.

If we want to study this definition in $\ZF$ we have to add the assumption of $\DC$, since the existence of countable elementary submodels is equivalent to $\DC$ itself. The first author and David Asper\'o studied properness in $\ZF$ in \cite{AsperoKaragila:2020} and argued that at least in the presence of $\ZF+\DC$, Mekler's definition is a good way to approach c.c.c. We support this, as shown in the following proposition.

\begin{proposition}\label{prop:Mekler}
  $\ZF+\DC$ proves that $\ccc3$ is equivalent to Mekler's definition.
\end{proposition}
\begin{proof}
  Suppose that $\PP$ is $\ccc3$, and let $M$ be a countable elementary submodel of some $H_\kappa$.\footnote{In $\ZF$ we define $H_\kappa$ as the set of those $x$ whose transitive closure does not map onto $\kappa$. See \cite{AsperoKaragila:2020} for more details.} Then, for every predense $D\in M$, there is some countable predense $D_0\subseteq D$, and by elementarity we have such $D_0\in M$ as well. Moreover, by countability and the fact that $\omega\in M$, we have that $D_0\subseteq M$ and so $D_0\cap M=D_0$. Therefore $D\cap M$ is predense in $\PP$, so $\1_\PP$ is $M$-generic.

  In the other direction, suppose that $\PP$ is c.c.c.\ by Mekler's definition, and let $D$ be a predense set. Let $M$ be a countable elementary submodel of some $H_\kappa$ such that $\PP,D\in M$. Then $D\cap M$ is predense below $\1_\PP$, as per the definition of $M$-genericity, and therefore $D\cap M$ is a countable predense subset of $D$.
\end{proof}

\subsection{Bukovsk\'y's c.c.c.: every new function is a choice function}
Bukovsk\'y proved the following theorem characterising when a model of $\ZFC$, $W$, is a $\kappa$-.c.c.\ extension of another model of $\ZFC$, $V$.
\begin{theorem*}[Bukovsk\'y, \cite{Bukovsky:1973}, also \cite{Bukovsky:2017}]
  Suppose that $V\subseteq W$ are models of $\ZFC$. Then $W$ is a generic extension of $V$ by a $\kappa$-c.c.\ forcing if and only if for every $x\in V$ and $f\colon x\to V$ in $W$, there is some $g\colon x\to V$ in $V$ such that:
  \begin{enumerate}
  \item $V\models|g(u)|<\kappa$ for all $u\in x$, and
  \item $W\models f(u)\in g(u)$ for all $u\in x$.
  \end{enumerate}
\end{theorem*}
This definition is often translated to a covering and approximation properties used extensively in the study of set theoretic geology, as they play a significant role in the ground model definability theorem (see Theorem~3 in \cite{Laver:2007}).

\begin{proposition}
$\ZF$ proves that if $\PP$ is $\ccc3$ then for every $V$-generic filter $G$, Bukovsk\'y's condition holds for $V\subseteq V[G]$ with $\kappa=\aleph_1$.
\end{proposition}
\begin{proof}
  Suppose that $\PP$ is $\ccc3$ and let $f\colon x\to y$ be some function in $V[G]$ with $x,y\in V$. Let $\dot f$ be a name such that $\dot f^G=f$, and let $p\in G$ be a condition forcing that $\dot f\colon\check x\to\check y$ is a function. For each $u\in x$ let $D_u$ be the set of conditions which decide the value of $\dot f(\check u)$ and let $E_u=\{w\in y\mid\exists p\in D_u, p\forces\dot f(\check u)=\check w\}$. Since $D_u$ is a dense open set and $\PP$ is $\ccc3$, it contains a countable predense subset, and therefore $E_u$ must be a countable set for all $u\in x$. Thus, setting $g(u)=E_u$ shows that Bukovsk\'y's condition holds.
\end{proof}

It is not immediately clear that the other direction holds as well, i.e.\ Bukovsk\'y's definition is equivalent to $\ccc3$, and we suspect that a positive answer can be found by an enthusiastic graduate student interested in forcing over models of $\ZF$. We will direct this student to Bukovsk\'y's paper, noting that the proof of Theorem~3 in \cite{Bukovsky:2017}, which is used to show that $W$ is indeed a generic extension of $V$, actually holds in $\ZF$. We do have an important corollary of the above proposition.

\begin{corollary}\label{cor:ccc3-collapses}
If $\PP$ satisfies Bukovsk\'y's condition with $\kappa=\aleph_1$, then $\PP$ does not change cofinalities at all and does not collapse any $\kappa>\omega_1$.\footnote{We note that preserving cofinalities might be distinct from preserving cardinals in $\ZF$. For example, it is consistent that $\cf(\kappa)=\cf(\kappa^+)=\kappa$, in that case forcing with $\Col(\kappa,\kappa^+)$ will collapse cardinals but not change cofinalities.} In particular, if $\omega_1$ is regular, then it is not collapsed.
\end{corollary}
To prove this, we need a small combinatorial lemma.
\begin{lemma}\label{lemma:countable-unions}
If $\mu\geq\omega_1$ and $\{A_\alpha\mid\alpha<\mu\}$ is a family of countable sets of ordinals, then $|\bigcup\{A_\alpha\mid\alpha<\mu\}|\leq\mu$. Moreover, if $\omega_1$ is regular, this holds for $\mu=\omega$ as well.
\end{lemma}
\begin{proof}[Proof of \autoref{cor:ccc3-collapses}]
  This is, in a nutshell the usual proof in $\ZFC$ that a c.c.c.\ forcing preserves cofinalities, but with a small twist at the end. Let $\PP$ be a $\ccc3$ forcing, and let $G\subseteq\PP$ be a $V$-generic filter.

If $f\colon\mu\to\kappa$ is any function in $V[G]$, then by Bukovsk\'y's condition there exists some $F\colon\mu\to[\kappa]^{<\omega_1}$ in $V$ such that $f(\xi)\in F(\xi)$ for all $\xi<\mu$. If $f$ was a bijection in $V[G]$, then $\bigcup\{F(\xi)\mid\xi<\mu\}$ must cover $\kappa$, but by the lemma it has size $\mu$, so $\mu=\kappa$. If $f$ was cofinal, then we may assume that $\kappa$ is regular, otherwise it defines a cofinal function in $\cf(\kappa)$ instead,\footnote{Unless $\cf(\kappa)=\omega$, in which case we cannot change the cofinality of $\kappa$ by adding new sets to the universe.} in this case $\xi\mapsto\sup F(\xi)$ is a cofinal function in $\kappa$, so again $\mu=\kappa$.
\end{proof}
\begin{proof}[Proof of \autoref{lemma:countable-unions}]
  Since $\{A_\alpha\mid\alpha<\mu\}$ are all sets of ordinals, the union can be enumerated as $\{a_\xi\mid\xi<\lambda\}$ for some $\lambda$. We define an injection from $\lambda$ into $\mu\times\omega_1$: $\xi\mapsto\tup{\alpha,\beta}$ if and only if $\alpha$ is the least such that $x_\xi\in A_\alpha$, and $\beta$ is the position of $a_\xi$ in $A_\alpha$ when considering the canonical enumeration of $A_\alpha$.

  It is easy to see that this function is injective, but since $\mu\cdot\aleph_1=\mu$ we get that $\lambda\leq\mu$ as wanted. If $\omega_1$ is regular, then this holds also for the case that $\mu=\omega$ since that implies that the range of the above function is bounded in $\omega_1$.
\end{proof}
\begin{corollary}\label{cor:ccc3-is-well-behaved}
  If $\kappa>\omega_1$, then a $\ccc3$ forcing cannot or change cofinalities or collapse any $\kappa>\omega_1$, and if $\omega_1$ is regular then a $\ccc3$ forcing does not change cofinalities or collapse cardinals at all.\qed
\end{corollary}
\subsection{Instances where the Main Theorem fails in \texorpdfstring{$\ZF$}{ZF}}
Our counterexamples in the previous section can be seen as somewhat ridiculous. Indeed, every counterexample was $\ccc1/\ccc2$ for some silly reason, like the lack of uncountable antichains altogether. But forcing works in a way that allows us to replace any preordered set by a unique complete Boolean algebra through a sequence of steps: quotient to get a partial order, quotient again to get a separative order, then take the Boolean completion. All of these steps are very explicit and so can be done without utilising the axiom of choice at all.

\begin{proposition}\label{prop:cBa}
If $\PP$ is a complete Boolean algebra, then $\ccc1(\PP)\to\ccc2(\PP)$.
\end{proposition}
\begin{proof}
Given an antichain $A$ in any complete Boolean algebra, $\PP$, $A\cup\{\sup A^\perp\}$ is a maximal antichain, so by $\ccc1(\PP)$ it is also countable, and so must $A$ be countable as well.
\end{proof}
Nevertheless, the counterexample in \autoref{thm:ccc2} still works, since it requires shrinking a given predense set, and it is not clear if $\ccc3$ follows from $\ccc2$+``$\PP$ is a complete Boolean algebra''. So even in the Boolean completion of that partial order we can take the same predense set and use it as a counterexample. This once again reinforces the observation that $\ccc3$ is somehow ``the correct way'' of thinking about c.c.c.\ in $\ZF$.

\begin{proposition}
Suppose that $\PP$ is a $\ccc1$ complete Boolean algebra. Then $\PP$ satisfies Bukovsk\'y's condition with $\kappa=\aleph_1$.
\end{proposition}
\begin{proof}
Suppose that $G\subseteq\PP$ is a $V$-generic filter and $f\in V[G]$, $f\colon\omega\to V$, and let $\dot f$ be a name for $f$. For each $n<\omega$ and $x\in V$, let $D_{n,x}=\{p\in\PP\mid p\forces\dot f(\check n)=\check x\}$. Since $\PP$ is a complete Boolean algebra, there is a condition $p_{n,x}=\sup D_{n,x}$ in $\PP$. Moreover, $A_n=\{p_{n,x}\mid x\in V\}$ is a maximal antichain in $\PP$ and therefore it is countable. We simply define $F(n)=\{x\mid p_{n,x}\in A_n\}$, and then $F$ witnesses that Bukovsk\'y's condition holds as wanted.
\end{proof}

Finally, many times our forcing has a natural tree structure (when reversing the order), we say that a forcing $\PP$ is \textit{co-well founded} it is well-founded in the reverse order.
\begin{proposition}
If $\PP$ is co-well founded and $\ccc1$, then it is $\ccc3$.
\end{proposition}
\begin{proof}
Let $r$ be a rank function on $\PP$ with the reverse order, and let $D$ be a predense set. Then the set of rank-minimal members of $D$ is an antichain and it is predense, therefore it is a maximal antichain, so by $\ccc1$ it is countable.
\end{proof}
This means that in many natural examples we can think of c.c.c.\ just as we did in $\ZFC$, as many of the instances of forcing notions we care about have a natural rank function. For example, $X^{<\omega}$, ordered by $\supseteq$, for any set $X$.

Both of these conditions lend themselves to the following corollary.
\begin{corollary}\label{cor:cba-co-wf-are-nice}
If $\PP$ is a $\ccc1$ complete Boolean algebra or is co-well founded, then the conclusion of \autoref{cor:ccc3-is-well-behaved} holds for $\PP$.\qed
\end{corollary}

\subsection{Preservation of Dependent Choice Principles}
\begin{theorem}
Suppose that $\ZF+\DC_\kappa$ holds and that $\PP$ is a $\ccc3$ forcing, then $\forces_\PP\DC_\kappa$.
\end{theorem}
\begin{proof}
  For $\kappa=\omega$, the work of Asper\'o and the first author in \cite{AsperoKaragila:2020} showed that proper forcing must preserve $\DC$, and by \autoref{prop:Mekler} we get that $\ccc3$ forcing notions are proper, assuming $\ZF+\DC$. So the conclusion holds. It is therefore enough to assume that $\kappa$ is uncountable.

  \begin{lemma}\label{lemma:DC-antichains-dense}
    Assuming $\DC_{\omega_1}$, if $\PP$ is a $\ccc3$ forcing, then every dense open set contains a maximal antichain.
  \end{lemma}
  \begin{proof}[Proof of Lemma]
    This is similar to the proof of \autoref{prop:ccc2-to-ccc3-under-strong-dc}. We recursively construct an antichain inside the dense open set, and the process must stabilise at a countable stage, and the antichain must therefore be maximal.
  \end{proof}

  Let $\PP$ be a $\ccc3$ forcing and let $\dot T$ be a $\PP$-name such that $\1_\PP\forces``\dot T$ is $\check\kappa$-closed without maximal nodes''. We will construct a name for a sequence of length $\kappa$ by recursion. Suppose that for all $\beta<\alpha$, $\dot t_\beta$ was defined and $\1_\PP\forces``\tup{\dot t_\beta\mid\beta<\alpha}^\bullet$ is an increasing sequence in $\dot T$''.

  Consider the set $D=\{p\mid \exists\dot t\,p\forces\forall\beta<\alpha:\dot t_\beta<_T\dot t\}$, by the fact that $\dot T$ is forced to be $\kappa$-closed and $\alpha<\kappa$, $D$ is in fact dense and open. By \autoref{lemma:DC-antichains-dense}, $D$ contains a maximal antichain $D_0$. For each $p\in D_0$, let $\dot s_p$ be a name witnessing that $p\in D$. Note that since we assume $\DC_\kappa$, we can certainly make these countably many choices. Let $\dot t_\alpha$ be a name such that $p\forces\dot t_\alpha=\dot s_p$ for all $p\in D_0$. By the maximality of $D_0$ we get that $\1_\PP\forces\dot t_\beta<_T\dot t_\alpha$ for all $\beta<\alpha$.

  Finally, by assuming $\DC_\kappa$ holds in the ground model, this construction generates a sequence of length $\kappa$ as wanted.
\end{proof}

We note here that for $\kappa>\omega$, we know that $\ccc2$ is equivalent to $\ccc3$, so it follows that $\ccc2$ will also preserve $\DC_\kappa$ in that case.

\begin{proposition}
It is consistent that a $\ccc2$ forcing can violate $\DC$.
\end{proposition}
\begin{proof}
  As we will see in \autoref{thm:dc-ccc2-collapse}, it is consistent with $\ZF+\DC$ that a $\ccc2$ forcing can collapse $\omega_1$. We can force, while adding such a collapsing counterexample, and also add a family of sets $\cA=\{A_\alpha\mid\alpha<\omega_1\}$ which does not admit a choice function, and will not have one added once we collapse $\omega_1$.

  Once $\omega_1$ is collapsed, the family $\cA$ will become countable and will witness that $\DC$, and in fact countable choice, fails.
\end{proof}
\section{Collapsing cardinals with ccc forcings}
\begin{theorem}[Folklore]\label{thm:folklore}
  $\DC\iff$No $\sigma$-closed forcing collapses $\omega_1$.
\end{theorem}
\begin{proof}
  If $\DC$ holds, then a $\sigma$-closed forcing does not add new reals, in particular it cannot collapse $\omega_1$.\footnote{This is a folklore result that $\DC$ is equivalent to the statement ``Every $\sigma$-closed forcing does not add $\omega$-sequences of ground model elements'', see Theorem~2.1 in \cite{Karagila:DC}. One can also use properness, under $\DC$, by simply generalising the proof of Proposition~4.2 in \cite{AsperoKaragila:2020} and applying Proposition~4.4.}

  Suppose $\DC$ fails, and let $T$ witness this. Namely, $T$ is a tree of height $\omega$ without maximal nodes and without infinite branches.

  Define the forcing $\{\tup{f,t}\mid f\in\Col(\omega,\omega_1), t\in T,\dom t=\dom f\}$ with the order $\tup{f,t}\leq\tup{f',t'}$ if and only if $f'\subseteq f$ and $t$ extends $t'$ in $T$. Here $\Col(\omega,X)$ is the set of all partial functions from $\omega$ to $X$, ordered by $\supseteq$, and its generic defines a surjection from $\omega$ onto $X$.

  Easily, this forcing collapses $\omega_1$: if $G$ is $V$-generic, $\bigcup\{f\mid\exists t,\tup{f,t}\in G\}$ is a generic for $\Col(\omega,\omega_1)$. On the other hand, if $\tup{f_n,t_n}$ is a decreasing sequence, then there is some $n$ such that $t_m=t_n$ for all large enough $m$, otherwise we have a branch in $T$. Therefore $\dom f_m=\dom t_m=\dom t_n$, but then it is necessarily the case that $f_n=f_m$ for all large enough $m$ as well.
\end{proof}
\begin{proposition}\label{prop:ccc1-maker}
Suppose that there is a forcing $\CC$ without maximal antichains (so it is necessarily $\ccc1$, but quite possibly $\ccc2$), then for every forcing $\QQ$ there is a $\ccc1$ forcing which adds a $V$-generic filter for $\QQ$.
\end{proposition}
\begin{proof}
  Consider the forcing $\PP$ given by $\{\tup{q,c}\mid q\in\QQ,c\in\CC\setminus\{\1_\CC\}\}$ with the order given by $\tup{q,c}\leq\tup{q',c'}$ if and only if $q\leq q'$ and $c\leq c'$ in the respective forcings. If $A$ is an antichain in $\PP$, then for every $q\in\QQ$, $A_q=\{c\mid\tup{q,c}\in A\}$ is an antichain in $\CC$. Since $\CC$ does not have any maximal antichains, there is some $c'$ such that $A_q\cup\{c'\}$ is still an antichain in $\CC$. But easily, $A\cup\{\tup{q,c'}\}$ is an antichain, and therefore $A$ was not maximal to begin with.
\end{proof}
\begin{corollary}
If there is a forcing without maximal antichains, then there is a $\ccc1$ forcing which collapses $\omega_1$. It is also consistent that this forcing is also $\sigma$-closed.
\end{corollary}
\begin{proof}
The first part is an immediate corollary of \autoref{prop:ccc1-maker}. The second part follows by considering the proof of \autoref{thm:ccc1} defining $\CC$ as the symmetric copy using finite supports and noting that no infinite decreasing sequence is symmetric. The same argument as \autoref{thm:folklore} works to show that there are no decreasing sequences in $\PP$ as defined in \autoref{prop:ccc1-maker}.
\end{proof}

Arguably, $\ccc1$ is not a reasonable definition for the countable chain condition anyway, rendering the above not much more than a curiosity. However, we can get a significantly more interesting result: it is consistent with $\ZF+\DC$ that a $\ccc2$ forcing collapses $\omega_1$.
\begin{theorem}\label{thm:dc-ccc2-collapse}
It is consistent with $\ZF+\DC$ that there exists a $\ccc2$ forcing which collapses cardinals.
\end{theorem}
\begin{proof}
  Starting from a model of $\ZFC$, let $\BB$ be an uncountable, productively c.c.c.\ forcing which is homogeneous and does not have a maximum, e.g.\ the Boolean completion of $\Add(\omega,\omega_1)$ with its maximum removed.

  Consider the lottery sum $\bigoplus_{\tup{n,\alpha}\in\omega\times\omega_1}\BB_{n,\alpha}$, where $\BB_{n,\alpha}=\BB$ for all $\tup{n,\alpha}$. We take the automorphism group which acts on each summand individually, and with the ideal of subsets given by unions over countably many summands at a time. Applying \autoref{thm:transfer} we get a symmetric copy of this partial order that has the property that the index of the summation is still $\omega\times\omega_1$, since we are not allowed to move the indices. This means that we can still discuss $\BB_{n,\alpha}$ as the $\tup{n,\alpha}$th summand in the symmetric copy as well.

  By homogeneity, any subset $A$ of the symmetric copy which meets uncountably many summands simultaneously will have at most countably many non-trivial intersections. This is similar to the construction in \autoref{thm:ccc2}.

  In addition we may assume that $\DC$ holds in the symmetric extension, since the ideal of subsets is certainly countably closed. We may also assume that $\BB$ itself does not have any new subsets as well. Working in this symmetric extension consider the partial order $\PP$ given by pairs $\tup{t,\vec b}$ such that:
  \begin{enumerate}
  \item $t\in\omega_1^{<\omega}$ and $\dom t=n$.
  \item $\vec b=\tup{b_0,\dots,b_{n-1}}$ and $b_i\in\BB_{i,t(i)}\}$.
  \end{enumerate}
  We define the order on $\PP$ as follows: $\tup{t,\vec b}\leq\tup{t',\vec b'}$ if and only if all the following conditions hold:
  \begin{enumerate}
  \item $t'=t\restriction\dom t'$, i.e.\ $t$ extends $t'$ in the forcing given by $\omega_1^{<\omega}$.
  \item For all $i<\dom t'$, $b_i\leq_{n,\alpha} b'_i$.
  \end{enumerate}

  In other words, we are describing an iteration: first add a surjection from $\omega$ onto $\omega_1$ using finite conditions, $f$, then consider the product $\prod_{\tup{n,\alpha}}\BB_{n,\alpha}$. 

  If $G\subseteq\PP$ is a generic filter, then $\bigcup_{\tup{t,\vec b}\in G}t$ is the surjection $f$ we mention. It remains to prove that every antichain in $\PP$ is countable.

  We denote by $\pi_{n,\alpha}$ the projection of $\PP$ to $\BB_{n,\alpha}$ and by $\pi_1$ the projection to $\omega_1^{<\omega}$. Suppose that $D$ is an uncountable subset of $\PP$,
  we now consider two cases.

  \textbf{Case 1:} There is some $t\in\omega_1^{<\omega}$ such that $\pi_1^{-1}(t)\cap D$ is uncountable.

  In this case we may assume that $D=\pi_1^{-1}(t)\cap D$. Therefore, as a partial order, $D$ is a subset of $\{t\}\times\prod_{i\in\dom t}\BB_{i,t(i)}\cong\prod_{i\in\dom t}\BB$. Since $\BB$ was productively c.c.c.\, $D$ must contain two compatible elements.

  \textbf{Case 2:} There is no such $t$. In this case $\pi_1(D)$ is an uncountable subset of $\omega_1^{<\omega}$, so there is some $n+1$ such that $\pi_1(D)\cap\omega_1^{n+1}$ is uncountable. We may assume that $\pi_1^{-1}(\omega_1^{n+1})\cap D=D$ in this case, so if $\tup{t,\vec b}\in D$, then $\dom t=n+1$. We consider now $\pi_{n,t(n)}(D)$ for all $t\in\pi_1(D)$. This is an uncountable family, so it must be that co-countably many of these are unions of orbits, and in particular these orbits must be uncountable. But this is impossible, since that means that if $\tup{t,\vec b}\in D$ is such condition, then $\pi^{-1}_1(t)\cap D$ is uncountable, and we assume there is no such $t$. Therefore Case 2 is impossible.

  This shows that $\PP$ is $\ccc2$, as wanted. 
\end{proof}
We combine the above with \autoref{cor:cba-co-wf-are-nice} to obtain the following conclusion.
\begin{corollary}
$\ZF+\DC$ cannot prove that a $\ccc2$ partial order will have a $\ccc2$ Boolean completion.\qed
\end{corollary}
\section{Remarks on higher chain conditions}
As we suggest in the introduction, we chose to focus on c.c.c.\ as that is the most commonly used chain condition. Unlike c.c.c., general $\kappa$-c.c.\ forcings do not have a nice iteration theorem, so analogues of Martin's Axiom are either provably false or must be restricted to fairly weak families of forcings.\footnote{That is not to say that there is no interest in higher forcing axioms. It is just that they cannot be based on chain conditions alone.} Nevertheless, $\kappa$-c.c.\ forcings still commonly pop up throughout set theory.

All the constructions that we have done in and around the various $\ccc{i}$ definitions can be done with $\kappa$-c.c.\ for any regular $\kappa$. If we also add $\DC_{<\kappa}$, which we most certainly can as suggested in \autoref{thm:transfer}, then our forcings at least somewhat well-behaved.

However, in general forcing over models of $\ZF$ tends to wreak havoc, or perhaps bring order, to the cardinal structure in somewhat unintended ways. For example, forcing with $\omega_2^{<\omega_2}$ is a forcing that will add no new subsets of $\omega$ when starting with a model of $\ZFC$, but if we force with this forcing over a model where $\power(\omega)$ is a countable union of countable sets, then we will invariably add many reals and collapse $\omega_1$, since we will necessarily add a well-ordering of the ground model's $\power(\omega)$.

Since $\DC_{<\kappa}$ is equivalent to the statement ``$\kappa$-closed forcing is $\kappa$-distributive'',\footnote{This is a folklore theorem; a simple proof can be found in \cite{Karagila:DC}.} we know that if we limit ourselves to $\kappa$-closed and $\kappa$-c.c.\ forcings, then statements related to sets of ordinals (e.g., collapsing $\kappa^+$) will be have reasonable proofs. Because every forcing is at least $\omega$-closed by definition and $\DC_{<\omega}$ is a theorem of $\ZF$, this is one more reason to focus on the case of $\kappa=\aleph_1$.
\section{Conclusions and open questions}
We hope that we managed to convince the reader at this point that the ``correct'' definition of the countable chain condition in $\ZF$ is $\ccc3$. Both due to its power in preserving cofinalities under relatively mild assumptions, and the failure to do that very same thing for other definitions of c.c.c. For example, in \cite{AsperoKaragila:2020} the subject of proper forcing is studied in $\ZF+\DC$. One of the observations there is that a proper forcing does not collapse $\omega_1$.\footnote{This is embedded in the proof of Proposition~4.2 there.}
\begin{corollary}
$\ZF+\DC$ does not prove that a $\ccc2$ forcing is proper.
\end{corollary}
In contrast, as the proof in \autoref{prop:Mekler} shows, $\ZF+\DC$ readily shows that a $\ccc3$ forcing is proper.

We conclude our paper with a few open questions and conjectures which we hope will motivate researchers, young and old, to expand our knowledge about the mechanism of forcing in $\ZF$.

\begin{question}
Is it consistent with $\ZF$ that a $\ccc3$ forcing can collapse $\omega_1$?
\end{question}
\begin{question}
Does complete Boolean algebra + $\ccc1$ imply $\ccc3$ in $\ZF$? In $\ZF+\DC$?
\end{question}
Since complete Boolean algebras which are $\ccc1$ satisfy Bukovsk\'y's condition, the above would admit a positive answer if the following question is answered positively.
\begin{question}
Is Bukovsk\'y's condition with $\kappa=\aleph_1$ equivalent to $\ccc3$ in $\ZF$?
\end{question}
\begin{question}
What can we say about the preservation of non-$\aleph$ cardinals when forcing with a $\ccc3$ forcing?
\end{question}
The last question is of particular importance, since we know very little about how the non-$\aleph$ cardinals behave in generic extensions, and since a forcing can add new cardinals to models of $\ZF$, this opens the door to even more questions.\footnote{In \cite{Monro:1983} a generic extension of the Cohen model is given in which there is an amorphous set; since the Cohen model can be linearly ordered, it cannot contain an amorphous set. So a new cardinal is added by forcing.}

\begin{question}
What can be said about Martin's Axiom of the various versions of c.c.c.\ that were presented here, even in the case where we only require meeting countably many dense open sets? Clearly Martin's Axiom for $\ccc1$ implies Martin's Axiom for $\ccc3$. But do the reverse implications hold, or perhaps like \autoref{prop:cBa} the version for $\ccc3$ is separate from the others?
\end{question}
\begin{question}
Forcing axioms can be seen, in general, as stating that a certain class of forcings has an iteration theorem. In the case of Martin's Axiom this means finite support iteration. Does any of the versions of c.c.c.\ presented here have such a nice iteration theorem in $\ZF$ or even $\ZF+\DC$?
\end{question}
We note that Martin's Axiom was studied in $\ZF$, to some extent, by Shannon in \cite{Shannon:1990} and later by Tachtsis in \cite{Tachtsis:2016}. These papers concentrated on what we call $\ccc2$, which we see is quite weak, making Martin's Axiom perhaps a bit too strong. But as we point out, it is very unclear what happens when we harness the full power of Martin's axiom, even in the study of $\ccc3$.\medskip

  We can define another version of c.c.c.\ by stating ``Every antichain extends to a maximal countable antichain''. This readily implies $\ccc2$ and for a complete Boolean algebra the proof of \autoref{prop:cBa} shows that the two are equivalent. We can modify the proofs of \autoref{thm:ccc1} and \autoref{thm:ccc2} to show that this principle is indeed stronger than $\ccc2$ and that it does not imply $\ccc3$ in $\ZF+\DC$.
\begin{question}\label{qn:max}
  Where does this principle sit in the hierarchy of countable chain conditions? Can we have a forcing that satisfies this version of c.c.c.\ and collapses cardinals?
\end{question}
\bibliographystyle{amsplain}
\providecommand{\bysame}{\leavevmode\hbox to3em{\hrulefill}\thinspace}
\providecommand{\MR}{\relax\ifhmode\unskip\space\fi MR }
\providecommand{\MRhref}[2]{%
  \href{http://www.ams.org/mathscinet-getitem?mr=#1}{#2}
}
\providecommand{\href}[2]{#2}

\end{document}